\documentclass{article}
\usepackage{graphicx} 
\usepackage{color}
\usepackage{tikz}
\usepackage{amsmath}
\usepackage{todonotes}
\usepackage{amssymb}
\usepackage{amsthm}
\usepackage{esint}
\usepackage[msc-links,abbrev]{amsrefs}
\usepackage{hyperref}
\usepackage[capitalize]{cleveref}
\usepackage[english]{babel}
\usepackage{IEEEtrantools}
\usepackage{mathrsfs}
\usepackage{mathtools}
\usepackage{enumitem}

\newcommand\restr[2]{{
  \left.\kern-\nulldelimiterspace 
  #1 
  \vphantom{\big|} 
  \right|_{#2} 
  }}

\crefformat{equation}{(#2#1#3)}
\Crefformat{equation}{(#2#1#3)}

\def\sideremark#1{\ifvmode\leavevmode\fi\vadjust{\vbox to0pt{\vss
 \hbox to 0pt{\hskip\hsize\hskip1em
 \vbox{\hsize3cm\tiny\raggedright\pretolerance10000
 \noindent #1\hfill}\hss}\vbox to8pt{\vfil}\vss}}}

\newtheorem{theorem}{Theorem}[section]

\newtheorem{lemma}[theorem]{Lemma}

\theoremstyle{definition}
\newtheorem{definition}[theorem]{Definition}

\theoremstyle{remark}

\numberwithin{equation}{section}

\title{The Singular CR Yamabe Problem and Hausdorff Dimension}
\date{}

\author{Sagun Chanillo and Paul C. Yang\\
\textit{To Haim Br\'ezis in memoriam}}

\begin{document}

\maketitle
\begin{abstract}
In this paper we consider CR analogs of Huber's theorem for Riemann surfaces. We also investigate the developing map for CR structures that are spherical in the case of three dimensional CR manifolds and give conditions when this developing map is injective.
\end{abstract}

\section{The main theorem}

Huber \cite{Huber1957} in 1957 showed that a complete surface with finite total curvature, is conformally a closed surface with finitely many points deleted. This result has inspired a lot of extension in various geometric settings.. In conformal geometry, the celebrated work of Schoen-Yau \cite{SYdevelop}, \cite{Schoenbook} studied the developing map of a locally spherical manifold of positive Yamabe class. They showed that the developing map is one-to-one into the sphere, and the complement of the image has small Hausdorff dimension. As a consequence, the original manifold is a quotient of a Kleinian group. It is our purpose to consider further extensions to the CR setting.

We begin with a pseudo-hermitian manifold whose conformal Laplacian is a positive operator, a CR invariant notion. The CR conformal Laplacian
$$L_{\theta}= -\Delta_b + c_n R,
$$
is paired with the Webster scalar curvature in the associated curvature equation for conformal change of contact form
$\tilde \theta = u^2 \theta$,
$$
L_{\theta} u = {\tilde R} u^{Q+2/Q-2},
$$
where $Q=2n+2$ denotes the homogeneous dimension. Similar to the conformal case, the analysis of this equation becomes
non-trivial when the operator is positive. On the other hand, when the operator is positive, much is known about the geometry
and analysis of the manifold.  For example, in the case where the CR structure is locally spherical and the CR Yamabe invariant is
positive, the development map can be shown to be injective and the complement of its image is small \cite{CY}. In the following, we localize
the analysis to a neighborhood of infinity and derive an analytic result about the size at infinity, assuming the scalar curvature is largely positive, assuming suitable size control of the negative part of the scalar curvature. The strategy of our proof is to use potential theory as adapted to the CR case in the same spirit as potential theory was used in the conformal case in \cite{ma2022linear}. Higher conformal geometry versions of Huber's theorem for $Q$ curvature have also been proved in \cite{Chang2004} and \cite{ma2022linear}.

In the last section we prove a result on the injectivity of the developing map in three dimensions. Injectivity results for the developing map have been proved earlier in \cite{CY} in three dimensions. The authors there assume conditions on the minimal integrability exponent of the Green function for the conformal sub-laplacian. We will instead assume that the CR Paneitz operator is non-negative and the Yamabe constant is positive. The CR Paneitz operator is defined in \cite{Chanillo}, \cite{Malchiodi} and the Yamabe constant is defined in $(1)$ in \cite{Malchiodi} and also in \cite{Chanillo}.

\begin{theorem}
\label{theorem-1}
Let $\left(M^n, \theta, J\right)$ be a compact pseudo-hermitian manifold. Let $Q=2n+2$ be the homogeneous dimension. Let $S \subseteq M$ be a closed set. Let $D$ be a bounded open neighborhood of $S$. Let
$$
\hat{\theta}=u^{\frac{4}{Q-2}} \theta
$$
be a conformal contact form on $D\backslash S$ such that the corresponding pseudo-Hermitian metric is complete near $S$. Assume
$$
R^{-}(\hat{\theta}) \in L^{\frac{2 Q}{Q+2}}\left(D \backslash S, d V_\theta\right) \cap L^p\left(D \backslash S, d V_\theta\right)
$$
for some $p>Q / 2$, where $d V_\theta=\theta \wedge(d \theta)^n$; and $R^-
(\hat{\theta})=\max \{-R(\hat{\theta}), 0\}$, where $R(\hat{\theta})$ is the scalar curvature coming from the contact form $\hat{\theta}$. Then the Hausdorff dimension satisfies \[\dim_{\mathcal{H}}(S) \leqslant \frac{Q-2}{2}.\] The notion of Hausdorff measure and dimension is that in~\cite{bonfiglioli2007stratified}(cf. Chapter 13).
\end{theorem}

\noindent
\textbf{Remark:}
\noindent
In a recent paper, Guidi-Maalaoui-Martino \cite{GMM} construct contact forms of constant positive scalar curvature on $S^{2n+1}\setminus S^{2k+1}$ if $k< (n-2)/2$. This is a partial converse to our result.

\begin{lemma}
\label{lemma-1}
Let $\mu$ be a non-negative Radon measure on a complete pseudo- Hermitian manifold $(M, \theta, J)$. Assume $\mu$ is absolutely continuous with respect to the measure $\theta \wedge (d\theta)^n$. Define,
$$
G_d^{\infty}=\left\{x \in M: \overline{\lim} _{r \rightarrow 0} \mu\left(B_r(x)\right) r^{-d}=+\infty\right\}
$$
for any $d \in[0, Q)$. Then,
$$
\mathcal{H}_d\left(G_d^{\infty}\right)=0 \text {, }
$$
where $\mathcal{H}_d$ is the Hausdorff measure of dimension $d$.
\end{lemma}

\begin{proof}
We adapt a proof from the book by W. Ziemer ~\cite{ziemer}. Assume for every $x \in A$, where $A$ is some Borel set, we have $$
\overline{\lim }_{r \rightarrow 0}  \frac{\mu\left(B_r(x)\right)}{r^d}>\lambda>0
$$
Then $\exists C=C(d)$; such that
\begin{equation}
\label{eq:1}
\mu(A) \geq C \lambda \mathcal{H}_d(A)
\end{equation}
where $\mathcal{H}_d(A)$ is the $d$-dimensional Hausdorff measure of $A$.
We may assume $\mu(A)<+\infty$ or else there is nothing to prove. Choose $\varepsilon>0$. Let us choose $U$ open, $U\supset A$, $\mu(U)<+\infty$, which we can as $\mu$ is a Radon measure, and $A$ a Borel set. Let $\mathcal{G}$ be a family of closed balls $B_r(x) \subset U$ such that for $x \in A$
$$
 \frac{\mu(B_r(x))}{r^d}>\lambda, 0<r<\varepsilon / 2 .
$$
We may extract from $\mathcal{G}$ a disjoint sub-family $\mathcal{F}\subset \mathcal{G}$ such that if $ \mathcal{F}^*$ is a finite sub-family of $\mathcal{F}$, one has
\[A \subseteq \cup\left\{B, B \in \mathcal{F}^*\right\} \cup\left\{\widehat{B}: \widehat{B} \in \mathcal{F} \backslash \mathcal{F}^*\right\}\]
where $\hat{B}$ is a ball concentric with $B$, and radius 5 times that of $B$. This is proved as Cor. 1.3.3 ~\cite{ziemer} and the proof is unchanged for us. Thus by definition
$$
\mathcal{H}_d^{5 \varepsilon}(A) \leqslant \sum_{B \in \mathcal{F}^*}\left(\frac{\delta(B)}{2}\right)^d+C 5^d \sum_{B \in \mathcal{F} \backslash \mathcal{F}^*}\left(\frac{\delta(B)}{2}\right)^d
$$
where $\delta(B)$ is the diameter of $B$. Since $\mathcal{F}\subset \mathcal{G}$ and $\mathcal{F}$ is disjoint, we have
$$
\begin{aligned}
\sum_{B \in \mathcal{F}}\left(\frac{\delta(B)}{2}\right)^d & \leqslant C \lambda^{-1} \sum_{B \in \mathcal{F}} \mu(B) \\
& \leqslant C \lambda^{-1} \mu(U)<+\infty .
\end{aligned}
$$
Since the remainder $$
C 5^d \sum_{B \in \mathcal{F} \backslash \mathcal{F}^*}\left(\frac{\delta(B)}{2}\right)^d
$$
can be made arbitrarily small, we get
$$
\mathcal{H}_d^{5 \varepsilon}(A) \leqslant C \lambda^{-1} \mu(U) .
$$
By the definition of Hausdorff measure we obtain~\cref{eq:1}.

For some fixed point $x_0 \in M$, define for $\rho(x,x_0)$ the pseudo-hermitian metric,
$$
A_m=\left\{x: \rho\left(x, x_0\right)<m,\overline{ \lim} _{r \rightarrow 0}  \frac{\mu(B(x, r))}{r^d}>\frac{1}{m}\right\} .
$$ Then $A_m$ is bounded and so $\mu\left(A_m\right)<+\infty$.
Using \cref{eq:1}
$$
\mu\left(A_m\right) \geqslant \frac{c}{m} \mathcal{H}_d\left(A_m\right) .
$$
Thus, $\mathcal{H}_d\left(A_m\right)<+\infty$, and we conclude that $\mathcal{H}_Q\left(A_m\right)=0$, since $d<Q$, that is the $\theta \wedge(\theta)^n$ measure of $A_m$ is zero. As $\mu$ is a Radon measure, the absolute continuity of $\mu$, with respect to $ \theta \wedge(d \theta)^n$, implies $\mu\left(A_m\right)=0$. Hence, by \cref{eq:1} again, we have $\mathcal{H}_d\left(A_m\right)=0$.
Since
$$
A=\cup A_m
$$
we conclude $\mathcal{H}_d\left(A\right)=0$. Our lemma follows by setting $A=G_d^{\infty}$.
\end{proof}
We next prove an intermediate result on the Hausdorff dimension preparatory to proving our main theorem. To do so, and for the main theorem, we need:
\begin{lemma}
\label{lemma-2}
Let $(M, \theta, J)$ be a compact pseudo-hermitian manifold. Assume the background scalar curvature satisfies $R(\theta) \leqslant 0$. Let $S \subset M$ be a closed set and $D$ an open neighborhood of $S$, which is small, i.e. $\operatorname{vol}(D)<\varepsilon$. Let $\hat{\theta}=u^{\frac{4}{Q-2}} \theta$
be a conformal contact form in $D \backslash S$. Assume $R^{-}(\hat{\theta}) \in L^p\left(D \backslash S, d V_{\hat{\theta}}\right), p>\frac{Q}{2}$ and $R^{-}(\hat{\theta})=\max \{-R(\hat{\theta}), 0\}$. Then if the pseudo-hermitian manifold $(M, \hat{\theta}, J)$ is complete, we have
$$
u(x) \rightarrow+\infty \text {, as } x \rightarrow S .
$$
\end{lemma}

\noindent
\textbf{Remark:} The hypothesis $R(\theta) \leqslant 0$, can be dropped. However, in our application, we can arrange $R(\theta) \leqslant 0$ by taking connected sums of $M$ with manifolds with very negative Yamabe constants, so we assume this as it simplifies the proof. We use the Moser iteration argument in ~\cite{Chang2004} and ideas in~\cite{ma2022linear}.
\begin{proof}
We first observe for $f \in C_0^{\infty} (D\backslash S)$, we have
\begin{equation}
\label{eq:2.1}
\left(\int_{D \backslash S}|f|^{\frac{2 Q}{Q-2}} d V_{\hat{\theta}}\right)^{\frac{Q-2}{ Q}} \leqslant c\left\{\int_{D \backslash S}\left(\left|\nabla^{\hat{\theta}}_b f\right|^2+R(\hat{\theta})|f|^2\right) d V_{\hat{\theta}}\right\}
\end{equation} for some $c>0$.
The expression above is a conformal invariant and so we may replace $\hat{\theta}$ by $\theta$ and we simply show \cref{eq:2.1} for $\theta$. By the Sobolev inequality
\begin{equation}
\label{eq:2.2}
\left(\int_{D \backslash S}|f|^{\frac{2 Q}{Q \cdot 2}} d V_\theta\right)^{\frac{Q-2}{Q}} \leqslant c\int_{D \backslash S}\left|\nabla_b^\theta f\right|^2 d V_\theta
\end{equation}
We write the right side as
$$
\int_{D \backslash S}\left(\left|\nabla_b^\theta f\right|^2+R(\theta) | f|^2\right) d V_\theta-\int_{D \backslash S} R(\theta)|f|^2 d V_\theta,
$$
where
$$
\begin{aligned}
\left(\int_{D \backslash S} R(\theta)|f|^2 d V_\theta\right) &\leqslant \left(\int_{D \backslash S}|f|^{\frac{2 Q}{Q-2}} d V_\theta\right)^{\frac{Q-2}{Q}} \left(\int_{D \backslash S}|R|^{\frac{Q}{2}} d V_\theta\right)^{2 / Q} \\
& \leqslant \varepsilon\left(\int_{D \backslash S}|f|^{\frac{2 Q}{Q-2}} d V_\theta\right)^{\frac{Q-2}{Q}}.
\end{aligned}
$$
Thus from \cref{eq:2.2}, we get
$$
\begin{aligned}
& \left(\int_{D \backslash S}|f|^{\frac{2 Q}{Q-2}} d V_\theta\right)^{\frac{Q-2}{Q}} \leqslant c \int_{D\backslash S}\left(\left|\nabla_b^\theta f\right|^2+R|f|^2\right) d V_\theta +\varepsilon\left(\int_{D\backslash S}|f|^{\frac{2 Q}{Q-2}}\,dV_\theta\right)^{\frac{Q-2}{Q}} \\
&
\end{aligned}
$$ From the last inequality \cref{eq:2.1} follows.
Next write $\hat{\theta}=e^{\frac{(Q-2)}{2} w} \theta$. We re-write this as $\theta=e^{-\frac{(Q-2)} {2}w} \hat{\theta}=v^{\frac{4}{Q-2}} \hat{\theta}$. Thus,
$$
-c_n \Delta_b^{\hat{\theta}} v+R(\hat{\theta}) v=R(\theta) v^{\frac{Q-2}{\hat{Q}+2}} \leqslant 0 \text {. }
$$
This gives
\begin{equation}
\label{eq:2.3}
-c_u \Delta_b^{\hat{\theta}} v+R^{+}(\hat{\theta}) v \leqslant R^{-}(\hat{\theta}) v
\end{equation}
Our goal is to apply the Moser iteration method to \cref{eq:2.3} with \cref{eq:2.1}, and conclude a local $L^{\infty}$ bound on $v$, where $v^{\frac{4}{Q-2}}=e^{-\frac{Q-2}{2} w}$.
The $L^{\infty}$ bound we want to prove is
\begin{equation}
\label{eq:2.4}
\|v\|_{L^{\infty}(B_{1/2}^{\hat{\theta}}\left(x_0\right))}\leqslant c\left({\int}_{{B}_1^{\hat{\theta}}}|v|^{\frac{2 Q}{Q-2}} d V_{\hat{\theta}}\right)^{\frac{Q-2}{2 Q}}
\end{equation}
with the normalization
\begin{equation}
\label{eq:2.5}
\int_{B_1^{\hat{\theta}}\left(x_0\right)} d V_{\hat{\theta}}=1.
\end{equation}
Here $B_1^{\hat{\theta}}$ is a unit ball in the metric of $\hat{\theta}$. Now
$$
\begin{aligned}
|v|^{\frac{2 Q}{Q-2}}d V_{\hat{\theta}} & =e^{-\frac{(Q-2)^2}{8} \cdot \frac{2 Q}{Q-2} w} e^{\frac{(Q-2)}{2} \cdot \frac{Q}{2} w} \theta \wedge(d \theta)^n \\
& =\theta\wedge (d \theta)^n=d V_\theta .
\end{aligned}
$$
So from \cref{eq:2.4}
\begin{equation}
\label{eq:2.6}
\|v\|_{L^{\infty}\left(B_{1/2} ^{\hat{\theta}}(x_0)\right)} \leqslant c\left(\int_{B_1^{\hat{\theta}}\left(x_0\right)} d V_\theta\right)^{\frac{Q-2}{2 Q}}.
\end{equation}
But $\hat{\theta}$ gives rise to a complete metric. So if $x_0 \rightarrow S$, the volume of the ball $B_1^{\hat{\theta}}(x_0)$, as measured in the background metric $\theta$, must go to zero. So as $x \rightarrow S$, the right hand side of \cref{eq:2.6} must go to zero. Thus $v(x) \rightarrow 0$ as $x \rightarrow S$. This means $w(x) \rightarrow+\infty$ as $x \rightarrow S$. This proves our lemma if \cref{eq:2.4} holds. To prove \cref{eq:2.4}, we multiply \cref{eq:2.3} by $\eta^2 v^\beta$ where $0 \leqslant \eta \leqslant 1$, is a cutoff function supported in an open set $U \subset D \backslash S$, which we can ultimately take to be a unit ball.
Then we integrate with respect to $d V_{\hat{\theta}}$ with the normalization \cref{eq:2.5}. We get after integration by parts, and Cauchy-Schwartz:
\begin{align}
& c_n \beta \int_U \eta^2\left|\nabla_b^\theta v\right|^2 v^{\beta-1} d V_{\hat{\theta}}+\int_U R^{+}(\hat{\theta}) \eta^2 v^{\beta+1} d V_{\hat{\theta}}\nonumber \\
& \leqslant \int_U R^- (\hat{\theta})\eta^2 v^{\beta+1} d V_{\hat{\theta}}+\frac{c}{\beta} \int_U {|\nabla \eta|^2 v^{\beta+1} d V_{\hat{\theta}}} \nonumber
\end{align}
This yields,
\begin{multline}
\label{eq:2.7}
\frac{4 c_n \beta}{(\beta+1)^2} \int_U \eta^2\left|\nabla_b^{\hat{\theta}} v^{\frac{\beta+1}{2}}\right|^2 d V_{\hat{\theta}}+\int_U R^{+}(\hat{\theta}) \eta^2 v^{\beta+1} d V_{\hat{\theta}} \\
\leqslant \int R^{-}(\hat{\theta}) \eta^2 v^{\beta+1} d V_{\hat{\theta}} +\frac{c}{\beta} \int_U \left|\nabla\eta\right|^2 v^{\beta+1}d V_{\hat{\theta}}.
\end{multline}
Now
$$
\nabla_b^{\hat{\theta}}\left(\eta v^{\frac{\beta+1}{2}}\right)=\eta \nabla_b^{\hat{\theta}} v^{\frac{\beta+1}{2}}+v^{\frac{\beta+1}{2}} \nabla_b^{\hat{\theta}} \eta.
$$
Using above in \cref{eq:2.7} we get after rearranging terms,
$$
\begin{aligned}
& \frac{c_{n} \beta}{(\beta+1)^2} \int_U\left|\nabla_b^{\hat{\theta}}\left(\eta v^{\frac{\beta+1}{2}}\right)\right|^2dV_{\hat{\theta}}+\int_U R^{+}(\hat{\theta}) \eta^2 v^{{\beta+1}} d V_{\hat{\theta}} \\
& \leqslant \int_U \operatorname{R}^-(\hat{\theta}) \eta^2 v^{\beta+1} d V_{\hat{\theta}}+\frac{c \beta}{(\beta+1)^2} \int_u\left|\nabla_b^{\hat{\theta}} \eta\right|^2 v^{\beta+1} d V_{\hat{\theta}}
\end{aligned} $$

Multiplying by $\frac{(\beta+1)^2}{\beta}$, we get using $\beta\geq 1$,
$$ \begin{aligned}
&\int_U\left|\nabla_b^{\hat{\theta}}\left(\eta v^{\frac{\beta+1}{2}}\right)\right|^2 d V_{\hat{\theta}}+\frac{(\beta+1)^2}{\beta^2} \int_U R^{+}(\hat{\theta}) \eta^2 v^{\beta+1} d V_{\hat{\theta}}\leqslant \\
&
\end{aligned}
$$

\begin{equation}
\label{eq:2.8}
\frac{(\beta+1)^2}{\beta} \int_U R^{-}(\hat{\theta}) \eta^2 v^{\beta+1} d V_{\hat{\theta}}+\frac{(\beta+1)^2}{\beta^2} \int_U\left|\nabla_b^{\hat{\theta}} \eta\right|^2 v^{\beta+1} d V_{\hat{\theta}}
\end{equation}
Since
\begin{multline*}
\int_U\left|\nabla_b^{\hat{\theta}}\left(\eta v^{\frac{\beta+1}{2}}\right)\right|^2 d V_{\hat{\theta}}+\int_U R(\hat{\theta}) \eta^2 v^{\beta+1} d V_{\hat{\theta}} \\
\leqslant \int_U\left|\nabla_b^{\hat{\theta}}\left(\eta v^{\frac{\beta+1}{2}}\right)\right|^2 d V_{\hat{\theta}} + \int_U{R^+(\hat{\theta})\eta^2 v^{\beta+1}} d V_{\hat{\theta}},
\end{multline*}
on applying \cref{eq:2.1} to \cref{eq:2.8}, we have
\begin{align}
\left(\int_U\left|\eta v^{\frac{\beta+1}{2}}\right|^{\frac{2 Q}{Q-2}} d V_{\hat{\theta}}\right)^{\frac{Q-2}{Q}} \leqslant c & \int_U\left|\nabla_b ^{\hat{\theta}}\eta\right|^2 v^{\beta+1} d V_{\hat{\theta}} \nonumber\\
& +\frac{c(\beta+1)^2}{\beta} \int_U R^-(\hat{\theta}) \eta^2 v^{\beta+1} d V_{\hat{\theta}}\label{eq:2.9} .
\end{align}
We need a gain in $L^p$ norms to iterate via Moser, and for this we now use our hypothesis that $R^-(\hat{\theta}) \in L^p\left(D \backslash S, d V_{\hat{\theta}}\right)$, for some $p>Q / 2$. Using H\"older on the second term on the right in \cref{eq:2.9}
$$
\leqslant c \beta\left(\int_U R^{-}(\hat{\theta})^p d V_{\hat{\theta}}\right)^{\frac{1}{p}}\left(\int_U\left(\eta^2 v^{\beta+1}\right)^q d V_{\hat{\theta}}\right)^{1 / q}.
$$
Since $p>\frac{Q}{2}$, and $\frac{1}{p}+\frac{1}{q}=1$, it
follows that $q<\frac{Q}{Q-2}$. Thus by hypothesis, we have
\begin{align}
\left[\int_U\left(\eta v^{\frac{\beta+1}{2}}\right)^{\frac{2 Q}{Q-2}}\right]^{\frac{Q-2}{Q}} & \leqslant c \beta\left[\int_U\left(\eta v^{\frac{\beta+1}{2}}\right)^{2 q} d V_{\hat{\theta}}\right]^{\frac{1}{q}}\nonumber \\
& +\int_U\left|\nabla_b ^{\hat{\theta}} \eta\right|^2 v^{\beta+1} d V_{\hat{\theta}}.\label{eq:2.10}
\end{align}
Now by H\"older's inequality and the normalization in \cref{eq:2.5},
$$
\int_U\left|\nabla_b^{\hat{\theta}} \eta\right|^2 v^{\beta+1} d V_{\hat{\theta}} \leqslant\left[\int_U\left(\left|\nabla_b^{\hat{\theta}} \eta\right| v^{\frac{\beta+1}{2}}\right)^{2 q} d V_{\hat{\theta}}\right]^{1 / q} \text {. }
$$
Using the above in \cref{eq:2.10} and taking the square root we obtain
$$
\left[\int_U\left(\eta v^{\frac{\beta+1}{2}}\right)^{\frac{2 Q}{Q-2}}\right]^{\frac{Q-2}{2 Q}} \leqslant c \beta^{1 / 2}\left[\int_U\left[\left(\eta+\left|\nabla_b^{\hat{\theta}} \eta\right|\right) v^{\frac{\beta+1}{2}}\right]^{2 q} d V_{\hat{\theta}}\right]^{\frac{1}{2q}}
$$
with $2 q<\frac{2 Q}{Q-2}$. The gain in exponent allows the iteration to proceed to conclude \cref{eq:2.4}.
\end{proof}

We use the previous lemma to prove:
\begin{lemma}
\label{lemma-3}
Let $(M, \theta, J)$ be a compact pseudo-hermitian manifold with scalar curvature $R \leqslant 0$.
Let $S \subseteq M$ be a closed set and $D$ an open neighborhood of $S$. Let $\hat{\theta}=u^{\frac{4}{Q-2}} \theta$ be geodesically complete in $D\backslash S$. Assume that
$$
\begin{aligned}
& R^{-}(\hat{\theta}) \in L^{\frac{2Q}{Q+2}}\left(D \backslash S, d V_{\hat{\theta}}\right) \cap L^p\left(D \backslash S, d V_\theta\right) \\
& \text { for some } p>\frac{Q}{2} ; R^{-}(\hat{\theta})=\max \{-R(\hat{\theta}), 0\} \text {. }
\end{aligned}
$$
Then, the 2 -capacity of $S$ is zero and the Hausdorff dimension of $S$ satisfies
$$
\operatorname{dim}_{\mathcal{H}}(S) \leqslant Q-2 \text {. }
$$
Here the Hausdorff dimension and capacity is in the sense of~\cite{bonfiglioli2007stratified} (also see~\cite{Costea}).
\end{lemma}
\begin{proof}
Recall the scalar curvature equation
\begin{align}
\label{eq:2}
-c_u \Delta_b u=-R u+R^{+}(\hat{\theta}) u^{\frac{Q+2}{Q-2}}- & R^{-}(\hat{\theta}) u^{\frac{Q+2}{Q-2}}
\end{align}
in $D\backslash S$.
Now by H\"older,
\begin{multline}
\label{eq:3}
\int_{D\backslash S} R^{-}(\hat{\theta}) u^{\frac{Q+2}{Q-2}} d V_\theta \leqslant \left(\int_{D \backslash S} R^-(\hat{\theta})^{\frac{2 Q}{Q+2}} u^{\frac{2 Q}{Q-2}} d V_\theta\right)^{\frac{Q+2}{2 Q}} \operatorname{Vol}(D)^{\frac{Q-2}{2 Q}}<+\infty,
\end{multline}
where $\mathrm{Vol} (D)$ is with respect to $\theta$. The expression in \cref{eq:3} is finite by hypothesis. Next, under the hypothesis of our  \cref{lemma-2}, we have
$$
u(x) \rightarrow+\infty, x \rightarrow S \text {. }
$$
Using this fact above, define the test function
$$
u_{\alpha, \beta}= \begin{cases}\beta, & u \geqslant \alpha+\beta . \\ u-\alpha, & u<\alpha+\beta .\end{cases}
$$
and set
$$
\phi_{\alpha, \beta}=u_{\alpha, \beta}-\beta+\beta(1-\eta)
$$
where $\eta$ is a cut-off function that is equal to $1$ in a neighborhood of $S$. Define
$$
\Sigma_\alpha=\{x \in D: u(x)>\alpha\} \text {. }
$$
We note that if $\beta$ is sufficiently large, then $u_{\alpha, \beta} \in(0, \beta]$ in $\sum_\alpha$, and
$$
\begin{aligned}
& \phi_{\alpha, \beta}=0 \text { on }\{x \in D: u(x)=\alpha\} \cup\{x \in D: u \geqslant \alpha+\beta\} .
\end{aligned}
$$
Note that
$$
\nabla_b \phi_{\alpha, \beta}=\nabla_b u_{\alpha, \beta}-\beta \nabla_b \eta,\quad \nabla_b u=\nabla_b u_{\alpha, \beta}\text{ where }\nabla_b u_{\alpha, \beta} \neq 0.
$$
We multiply \cref{eq:2} by $\phi_{\alpha, \beta}$ and get after integration by parts, using as $d V_\theta$ the volume element $\theta \wedge(d \theta)^n$,
$$
\begin{aligned}
& c_n \int_{\sum_\alpha} \nabla_b u \cdot \nabla_b \phi_{\alpha, \beta} d V_\theta=\int_{\sum_{\alpha}}\phi_{\alpha,\beta}\left(-R u+R(\hat{\theta}) u^{\frac{Q+2}{Q-2}}\right) d V_\theta .
\end{aligned}
$$
\begin{multline}
\label{eq:4}
c_n\int_{\Sigma_a}\left|\nabla_b u_{\alpha, \beta}\right|^2 d V_\theta=\beta \int_{\sum_\alpha}\left(c_n \nabla_b u \cdot \nabla_b \eta+\left(-R u+R(\hat{\theta}) u^{\frac{Q+2}{Q-2}}\right)(1-\eta) d V_\theta \right. \\-\int_{\sum_\alpha}\left(-R u+R(\hat{\theta}) u^{\frac{Q+2}{Q-2}}\right)\left(\beta-u_{\alpha, \beta}\right) d V_\theta
+\int_{\sum_\alpha} R^{-}(\hat{\theta}) u^{\frac{Q+2}{Q-2}}\left(\beta-{u}_{\alpha, \beta}\right) d V_\theta .
\end{multline}
Using \cref{eq:3}, $R \leqslant 0, u_{\alpha, \beta} \in(0, \beta]$ and the location of the support of $1-\eta$, we easily see the right hand side of \cref{eq:4} is bounded by $C \beta$, where $C$ depends on $\alpha, \eta$ but not on $\beta$. Thus by \cref{lemma-2},
$$
\int_{\Sigma_\alpha}\left|\nabla_b\left(\frac{u_{\alpha, \beta}}{\beta}\right)\right|^2 \leqslant c / \beta \rightarrow 0,
$$
if $\beta \rightarrow \infty$. Thus $\operatorname{Cap}_2(S, D)=0$. Using Cor. 4.6 in ~\cite{Costea},  we conclude  the Hausdorff dimension of S satisfies
$\dim_{\mathcal H} S \leqslant Q-2$.
\end{proof}

Our goal is now to improve \cref{lemma-3}. To do so we first prove:
\begin{lemma}
\label{lemma-4}
Given that $u$ satisfies \cref{eq:2} the scalar curvature equation and the hypothesis on $R^{-}(\hat{\theta})$ of \cref{lemma-3} is satisfied and $R \leqslant 0$, where $R$ is the scalar curvature of the background contact form $\theta$. Then if $\hat{\theta}=u^{4 /(Q-2)} \theta$ is geodesically complete near $S \subseteq M, S$ closed, we have that $-\Delta_b u$ is a Radon measure and $-\left.{\Delta_b u}\right|_S \geqslant 0$. We restrict our attention in this lemma to $D$, an open neighborhood of $S$.
\end{lemma}
\begin{proof}
We begin by proving that if
$$
\begin{aligned}
-c_n \Delta_b u=-R u+R^{+}(\hat{\theta}) u^{\frac{Q+2}{Q-2}}-R^{-}(\hat{\theta}) u^{\frac{Q+2}{Q-2}}=f
\end{aligned}
$$
in $D\backslash S$, then $f \in L^{1}(D)$. Consider the function.
\[
\alpha_s(t) =
\begin{cases}
  t, & t \leqslant s, \\
  \text{{increasing}}, & t \in [s, 10s], \\
  2s, & t \geqslant 10s.
\end{cases}
\]
We observe that $\alpha_s^{\prime} \in[0,1], \alpha_s^{\prime \prime} \leqslant 0$, since $\alpha(s)$ can be selected to be concave. Then
\begin{equation}
\label{eq:5}
-\Delta_b \alpha_s(u)=-\alpha_s^{\prime \prime}(u)|\nabla u|^2+\alpha_s^{\prime}(u)\left(-\Delta_b u\right)
\end{equation}
Then for $s>\max \{u(x): x \in \partial D\}$, and by recalling that
$
\Delta_b=\sum_{i=1}^{2n} X_i^2
$, we have by Green's theorem and where $\nu$ is the outward unit normal,
$$
\int_D-\Delta_b \alpha_s(u) d V_\theta=-\sum_{i=1}^{2n}\int_{\partial D}\left\langle \nu, X_i\right\rangle \alpha_s(u) d\sigma
$$
On $\partial D, \alpha_s(u)=u$ by construction of $\alpha_s(t)$ and choice of $s$. Using \cref{eq:5} above, we get
\begin{multline}
\label{eq:6}
- \sum_{i=1}^{2n} c_n \int_{\partial D}\left\langle \nu, X_i\right\rangle u\, d \sigma =-\int_D c_n \alpha_s^{\prime \prime}(u)\left|\nabla_b u\right|^2 \\
+\int_D \alpha_s^{\prime}(u)\left[-R u-R(\hat{\theta}) u^{\frac{Q+2}{Q-2}}+R^{+}(\hat{\theta}) u^{\frac{Q+2}{Q-2}}\right] d V_\theta .
\end{multline}
Note that $\alpha_s^{\prime \prime} \leqslant 0$ by construction, and $-R u \geqslant 0$ by hypothesis. So from \cref{eq:6},
\begin{multline*}
\int_D \alpha_s^{\prime}(u)\left(R^{+}(\hat{\theta}) u^{\frac{Q+2}{Q-2}}-R u\right) \leqslant  -c_n \sum_{i=1}^{2n} \int\left\langle \nu, X_i\right\rangle u d \sigma
+\int_D R^{-}(\hat{\theta}) u^{\frac{Q+2}{Q-2}} d V_\theta .
\end{multline*}
By hypothesis we saw using the H\"older inequality in the previous lemma,
$$
\int_D R^{-}(\hat{\theta}) u^{\frac{Q+2}{Q-2}} d V_\theta<+\infty .
$$
Thus, from Fatou's lemma, letting $s \rightarrow+\infty$,
\begin{equation}
\label{eq:7}
\int_D\left(R^{+}(\hat{\theta}) u^{\frac{Q+2}{Q-2}}-R u\right) d V_\theta<+\infty .
\end{equation}
This proves $f \in L^{1}(D)$. Next, from \cref{eq:5} again, since $\alpha_s^{\prime \prime} \leqslant 0$,
\begin{multline*}
\int_D\left|\Delta_b \alpha_s(u)\right| \leqslant -c_n\int_D \alpha_s^{\prime \prime}(u)\left|\nabla_b u\right|^2 \\
+\int_D \alpha_s^{\prime}\left[-R u+R^{+}(\hat{\theta}) u^{\frac{Q+2}{Q-2}}+{R}^-(\hat{\theta}) u^{\frac{Q+2}{Q-2}}\right]dV_\theta
\end{multline*}
Using the identity \cref{eq:6} again and replacing the integral
$c_n \int_D-\alpha_s^{\prime \prime}(u)\left|\nabla_b u\right|^2$ above via \cref{eq:6},
\begin{multline*}
\int_D\left|\Delta_b \alpha_s(u)\right| \leqslant c_n \sum_{i=1}^{2n} \int_{\partial D}\left\langle \nu, X_i\right\rangle u\, d \sigma \\+2 \int_D \alpha_s^{\prime}(u)\left(R^{+}(\theta) u^{\frac{Q+2}{Q-2}}-R u\right) d V_\theta
+2 \int_D R^{-}(\hat{\theta}) u^{\frac{Q+2}{Q-2}} d V_\theta
\end{multline*}
Using \cref{eq:7} and \cref{eq:3} again, we get
$$
\int_D\left|\Delta_b \alpha_s(u)\right|<+\infty \text {. }
$$
This for $\phi \in C^0(D)$,
$$
\begin{aligned}
\left|\left\langle-\Delta_{b}\alpha_s(u), \phi\right\rangle\right| & =\left| \int_D \Delta\alpha_s(u) \phi \right| \\
& \leqslant\|\phi\|_{C^0(D)} \int_D\left|\Delta_b \alpha_s(u)\right| .
\end{aligned}
$$
We now have,
\begin{equation*}
\begin{aligned}
\left\langle\Delta_b u, \phi\right\rangle & =\int_D \nabla_b u \cdot \nabla_b \phi\, d V_\theta,\quad \phi \in C_0^{\prime}(D) \\
& =\lim _{s \rightarrow \infty} \int_D \alpha_s^{\prime}(u) \nabla_b \cdot \nabla_b \phi \,d V_\theta \\
& =\lim _{s \rightarrow \infty}\left\langle-\Delta_b \alpha_s(u), \phi\right\rangle.
\end{aligned}
\end{equation*}
Thus
\begin{multline*}
\left|\left\langle-\Delta_b u, \phi\right\rangle\right| \leqslant c\left\{\sum_{i=1}^{2n}\left|\int_{\partial D}\left\langle \nu, X_i\right\rangle u \,d \sigma\right|+\int_D R^{-}(\hat{\theta})^{\frac{Q+2}{Q-2}} u^{\frac{Q+2}{Q-2}} d V_\theta\right\}\|\phi\|_{c^0(D) .}
\end{multline*}
Thus $-\Delta_b u$ is a Radon measure.
We now show that it is a non-negative Radon measure. Let $\phi \geqslant 0, \phi \in C_0^{\infty}(D)$, then we have from above,
$$
\begin{aligned}
c_n\langle-\Delta u, \phi\rangle & =\lim _{s \rightarrow \infty} c_n \int_D\left(-\Delta \alpha_s(u)\right) \phi \,d V_\theta \\
& =\lim _{s \rightarrow \infty} \int_D-c_n \alpha_s^{\prime \prime}(u)\left|\nabla_b u\right|^2 \phi d V_\theta \\
& +\lim _{s \rightarrow \infty} \int_D \alpha_s^{\prime}(u)(-R u) \phi \,d V_\theta \\
& +\lim _{s \rightarrow \infty} \int_D \alpha_s^{\prime}(u) R(\hat{\theta}) u^{\frac{Q+2}{Q-2}} \phi \,d V_\theta
\end{aligned}
$$
We can drop the first term on the right above since $\alpha_s^{\prime\prime}\leq 0$ and so
$$
c_n\langle-\Delta u, \phi\rangle \geqslant-\left(\int_{\text{supp } \phi \backslash S} \mid-R u+R(\hat{\theta}) u^{\frac{Q+2}{Q-2}} \mid d V_\theta\right)
$$
The integral on the right hand side is finite from above, and so since
$$
\int_{\operatorname{supp }\phi\backslash S} d V_\theta \rightarrow 0 \text {, }
$$
with $\|\phi\|_{C^O(D)}=1$, we get $$
-\int_{\operatorname{supp} \phi\backslash S} \mid-R u+R(\hat{\theta}) u^{\frac{Q+2}{Q-2}} \mid d V_\theta \rightarrow 0
$$
and so $-\left.\Delta_b{ }u\right|_S \geqslant 0$.\end{proof}

\noindent
\textbf{Remark:} We note that $-\Delta u=f, f \in L^{1}(D)$, and so our Radon measure is non-negative and absolutely continues with respect to the volume measure $d V_\theta$.
Next from \cref{lemma-3}, $\operatorname{dim}_\mathcal{H}(S) \leqslant Q-2<Q$.
Thus \cref{lemma-1} applies.

Let $\Omega \subseteq M$ be a bounded, open set. Let $A$ be a non-negative Radon measure on $M$. We define for $\rho$ the metric distance,
\[
R_\mu^{\alpha, \Omega}(x)=\int_{\Omega} \frac{d \mu(y)}{\rho(x, y)^{Q-\alpha}} ; 0 < \alpha \leqslant Q .
\]
Let \( E \subseteq \Omega \). We define
\begin{definition}
The capacity
\[
C^\alpha(E, \Omega)=\inf \left\{\mu(\Omega) \mid R_\mu^{\alpha, \Omega}(x) \geqslant 1, \forall x \in E\right\}.
\]
\end{definition}
Let us look at a coordinate patch $U$ of $M$, and there we may assume that $U\cap M$ is Heisenberg with a dilation structure $\delta_\lambda$. We assume $E \subset \Omega \subset U$. Then
\begin{lemma}
\label{lemma-5}
Set $A_\lambda=\left\{\delta_\lambda(x): x \in A\right\}$.
Then
$$
C^\alpha\left(E_\lambda, \Omega_\lambda\right)=\lambda^{Q-\alpha} C^\alpha(E, \Omega), \alpha \in(1, Q] \text {. }
$$
\end{lemma}
\begin{proof}
For a non-negative Radon measure $\mu$, define a Radon measure $\mu^*$  associated to $\mu$ for all $A_\lambda\subset\Omega_\lambda$ by setting:
$$
\mu^*\left(A_\lambda\right)=\mu(A).
$$
Then
$$
\begin{aligned}
& R_{\mu^*}^{\alpha, \Omega_\lambda}\left(\delta_\lambda(x)\right)=\int_{\Omega_\lambda} \frac{d\mu^*(y)}{\rho\left(\delta_\lambda(x), y\right)^{Q-\alpha}} . \\
& =\int_{\Omega} \frac{d{\mu(y)}}{\rho\left(\delta_\lambda(x), \delta_\lambda(y)\right)^{Q-\alpha}}=\lambda^{\alpha-Q} \int_{\Omega} \frac{d\mu(y)}{\rho(x, y)^{Q-\alpha}} \\
& \\
& =\lambda^{\alpha-Q} R_\mu^{\alpha, \Omega}(x) .
\end{aligned}
$$
So,
$$
\begin{aligned}
C^\alpha\left(E_\lambda, \Omega_\lambda\right)&=\inf \left\{\mu^*\left(\Omega_\lambda\right): R_{\mu^*}^{\alpha,\Omega_\lambda}(\lambda x) \geq 1, \forall x \in E\right\} \\
& =\inf \left\{\mu(\Omega): \lambda^{\alpha-Q} R_\mu^{\alpha,\Omega}(x) \geqslant 1, \forall x \in E\right\} \\
& =\lambda^{Q-\alpha} \inf \left\{\lambda^{\alpha-Q} \mu(\Omega): \lambda^{\alpha-Q} R_\mu^{\alpha, \Omega}(x) \geqslant 1, \forall x \in E\right\} \\
&= \lambda^{Q-\alpha} C^\alpha(E, \Omega).\qedhere
\end{aligned}
$$
\end{proof}
\begin{lemma}
\label{lemma-6}
Let $x_0 \in M$ be a fixed point. Consider a metric ball $B_r\left(x_0\right)$. In this metric ball we can assume $M$ is Heisenberg if $r \leqslant r_0$ for $r_0$ small, and so we have scaling like in the previous lemma. Then for $\alpha>0$,
$$
C^\alpha\left(\partial B_2\left(x_0\right), B_{5 r}\left(x_0\right)\right)=c_0>0 .
$$
\end{lemma}
\begin{proof}
Let us take the surface measure on $\partial B_r\left(x_0\right)$, given by
$d\sigma=\theta\wedge e^1\wedge\cdots\wedge e^{2n-1}$, see ~\cite{Cheng2005}. Then for $x\in \partial B_r(x_0)$
$$
\begin{aligned}
& R_\sigma^{\alpha, B}(x) \geqslant m_0>0, \quad B=B_{5r}\left(x_0\right) . \\
= & \int_{\partial B_r\left(x_0\right)} \frac{d \sigma(y)}{\rho(x, y)^{Q-\alpha}} \geqslant m_0>0 .
\end{aligned}
$$
Thus by definition:
$$
C^\alpha\left(\partial B_r, B\right) \leqslant \frac{1}{m_0}<+\infty .
$$
Thus our capacity is finite. We want to prove that it is non-zero.
We claim that we can find $p \in \partial B_r\left(x_0\right)$ such that for any finite non-negative measure $\mu$ on $B_{5r}\left(x_0\right)$, we have
\begin{equation}
\label{eq:8}
\mu\left(B_s(p) \cap B_{5 r}\left(x_0\right)\right) \leqslant c_n \mu\left(B_{5 r}\left(x_0\right)\right) s^{Q-1}
\end{equation}
for all $s>0$.
Note that if $s$ is sufficiently large $s \geqslant 100$, then
\cref{eq:8} is obvious. Let us normalize $\mu\left(B_{5 r}\left(x_0\right)\right)=1$.
Assume \cref{eq:8} is false. Then for each $q \in \partial B_r\left(z_0\right)$, $\exists r_q>0$ such that
\begin{equation}
\label{eq:9}
\mu\left(B_{r_q}(q) \cap B_{5 r}\left(x_0\right)\right) \geqslant c_0 r_q^{Q-1} .
\end{equation}
Then $\left\{B_{r_q}(q)\right\}$ is a covering of $\partial B_r\left(x_0\right)$.
Using Vitali's lemma, we may find a disjoint sub-family of balls
$$
\left\{B_{r_q}\left(q_1\right), \ldots, B_{r_{q_k}\left(q_k\right)}\right\}, q_j \in \partial B_r\left(x_0\right),
$$
such that
$$
\left\{B_{3 r_1}\left(q_1\right), \ldots, B_{3 r_{q_1}}\left(r_{q_1}\right)\right\}
$$
covers $\partial B_r\left(x_0\right)$. The metric family of balls satisfy the doubling property (cf.~\cite{nagel1985balls}). By disjointness and by \cref{eq:9},
\begin{align}
\label{eq:10}
c_0 \sum_{j=1}^k r_{q_j}^{Q-1} & \leqslant \sum_{j=1}^k \mu\left(B_{r_{q_j}}\left(q_j\right) \cap B_{5 r}\left(x_0\right)\right) \leqslant \mu\left(B_{5 r}\left(x_0\right)\right)=1.
\end{align}
Now consider the surface measure introduced in~\cite{Cheng2005}. As the situation is locally Heisenberg, we observe $\theta \wedge e^1\wedge\cdots\wedge e^{2n-1}=d\sigma$ scales under Heisenberg scaling $\delta_\lambda$,
\begin{equation}
\label{eq:11}
\delta_\lambda\left(\theta \wedge e^1\wedge\cdots\wedge e^{2n-1}\right)=\lambda^{Q-1}\theta \wedge e^1\wedge\cdots\wedge e^{2n-1}.
\end{equation}
So
$$
\begin{aligned}
\left|\partial B_r\left(x_0\right)\right| & =\int_{\partial B_r(x_0)} d\sigma\\
& \leqslant \sum_{j=1}^k\left|B_{3 r_{q_j}}\left(q_j\right) \cap {\partial B_r\left(x_0\right)}\right|
\end{aligned}
$$
Using \cref{eq:11}, we get
$$
|\partial B_r (x_0)|\leqslant C \sum r_{q_j}^{Q-1}\left|\partial B_r\left(x_0\right)\right| \text {. }
$$
We have a contradiction to \cref{eq:10} if $c_0$ is large enough. So our claim is proved.
We can now prove our lemma. For $\alpha \in(1, Q)$, and $p$ chosen to satisfy \cref{eq:8},
$$
\begin{aligned}
& R_\mu^{\alpha, B_{5 r}\left(x_0\right)}(p)=\int \frac{d \mu(y)}{\rho(p, y)^{Q-\alpha}} \\
& =\sum_{j=-\infty}^{4}  \int_{\left\{2^j<\rho(p, y) \leqslant 2^{j+1}\right\} \cap B_{5 r}\left(x_0\right)} \frac{d \mu}{\rho(p, y)^{Q-\alpha}}+c_0 \mu\left(B_{5 r}\right) \\
&=\sum_{j=-\infty}^4 2^{-j(Q-\alpha)} \mu\left(B_{2^{j+1}}(p)\cap B_{5 r}\left(x_0\right)\right) +c_0 \mu\left(B_{5 r}\left(x_0\right)\right).
\end{aligned}
$$
Using \cref{eq:8}, we get
$$
R_\mu^{\alpha, B_{5 r}\left(x_0\right)}(p)\leqslant c\mu\left(B_{5 r}\left(x_0\right)\right) \sum_{j=-\infty}^4 2^{j(\alpha-1)}+c_0 \mu\left(B_{5 r}\left(x_0\right)\right) .
$$
Since $\alpha>1$,
$$
R_\mu^{\alpha,B_{5r}\left(x_0\right)}(p) \leqslant c_1 \mu\left(B_{5 r}\left(x_0\right)\right), \quad c_1>0 .
$$
Thus by definition, and recalling $\mu(B_{5r}(x_0))=1$, we have
$$
C^\alpha\left(\partial B_r, B_{5 r}\left(x_0\right)\right) \geqslant \frac{1}{c_1}>0.\qedhere
$$
\end{proof}
\begin{definition}
Let $E \subset M, p \in M$. We say $E$ is $\alpha$-thin at $p$ for $\alpha \in(1, Q)$ if
$$
\sum_{i \geqslant 1} \frac{C^\alpha\left(E \cap \omega_i^\delta(p), \Omega_i^\delta(p)\right)}{C^\alpha\left(\partial B_{2^{-i} \delta}(p), B_{2^{-i+1}(p)}(p)\right.}<+\infty,
$$
where $\delta>0$ is small and $$
\begin{aligned}
& \omega_i^\delta(p)=\left\{x: \rho(x, p) \in\left[2^{-i} \delta, 2^{-i+1} \delta\right]\right\}, \\
& \Omega_i^\delta(p)=\left\{x: \rho(x, p) \in\left(2^{-i-1} \delta, 2^{-i+2} \delta\right)\right\}.
\end{aligned}
$$
\end{definition}
Thus $\Omega_i^\delta(p) \supset \omega_i^\delta(p)$ is the doubled annuli.

The last lemma that we need is
\begin{lemma}
\label{lemma-7}
Let $\left(M^n, \theta, J\right)$ be a complete pseudo-hermitian manifold, and let $\mu$ be a finite Radon measure absolutely continuous with respect to $\theta \wedge(d \theta)^n$ on a bounded domain $G \subset M^n$. Let
$S \subset G$, with $S$ compact. Assume that the Hausdorff dimension satisfies the condition $\dim_{\mathcal{H}}S>d$. Let $\alpha \in(1, Q)$ and $\alpha<Q-\alpha$. Then there exists $p \in S$ and $E \subset S$, that is $\alpha$-thin at $p$, such that
$$
\int_G \frac{d\mu(y)}{\rho(x, y)^{Q-\alpha}} \leqslant \frac{C}{\rho(x, p)^{Q-\alpha-d}}
$$
for some constant $C$, and for all $x \in B_\delta(p) \backslash E$, for some $\delta>0$.
\end{lemma}

\noindent
\textbf{Remark:} In \cref{lemma-4}, we showed that under the hypotheses, we have
$$
-\Delta_b u=f=\mu,
$$
and $f \in L^{1}(D)$, so in our application $\mu$ is indeed a finite Radon measure that is absolutely continuous with respect to $\theta \wedge(d \theta)^n$.
\begin{proof}
By hypothesis, there is some $\varepsilon>0$, such
$$
\mathcal{H}_{d+\varepsilon}(S)=+\infty \text {, }
$$
since the Hausdorff dimension of $S$ is greater than $d$.
Thus, by \cref{lemma-1}, there is a point $p \in S$ such that
$$
\varlimsup_{r \rightarrow 0} \frac{\mu\left(B_r(p)\right)}{r^{d+\varepsilon}} \leqslant C<+\infty .
$$
Thus,
\begin{equation}
\label{eq:12}
\mu\left(B_r(p)\right) \leqslant C r^{d+\varepsilon},
\end{equation}
where $r$ is small. Next
\begin{equation}
\label{eq:13}
R_\mu^{\alpha, \Omega}(x)=\int_{\Omega} \frac{d{\mu(y)}}{\rho(x, y)^{Q-\alpha} }.
\end{equation}
Let $i_0$ be fixed and $i\geq i_0+2$. We decompose \cref{eq:13}
\begin{align}
&= \int_{\Omega\backslash B_{2^{-i_0+2}\delta}} \frac{d \mu(y)}{\rho(x, y)^{Q-\alpha}} + \int_{B_{2^{-i_0+2}\delta}\backslash\Omega} \frac{d \mu(y)}{\rho(x, y)^{Q-\alpha}} + \int_{\Omega^i_\delta} \frac{d \mu(y)}{\rho(x, y)^{Q-\alpha}}. \label{eq:14} \\
&= \mathrm{I} + \mathrm{II} + \mathrm{III} \nonumber
\end{align}
where we recall
$$
\begin{array}{ll}
\Omega_\delta^i&=\{y: \left.\rho(p, y) \in\left(2^{-i-1} \delta, 2^{-i+2} \delta\right)\right\}, \\
\omega_\delta^i & =\left\{x: \rho(p, x) \in\left(2^{-i} \delta, 2^{-i+1} \delta\right)\right\}.
\end{array}
$$
Clearly for $x \in \omega_i^\delta \subseteq \Omega$,
\begin{align}
I \leqslant & \frac{C \mu(\Omega)}{\left(2^{-i_0} \delta\right)^{Q-\alpha}} \nonumber\frac{C\left(2^{-i} \delta\right)^{Q-\alpha-d}}{\left(2^{-i_0} \delta\right)^{Q-\alpha}} \frac{\mu(\Omega)}{\rho(p, x)^{Q-\alpha-d}}\nonumber \\
& \leqslant \frac{C(\alpha,\delta,i_0,Q)}{\rho(p, x)^{Q-\alpha-d}},\label{eq:15}
\end{align}
since $d<Q-\alpha$.
We estimate $\mathrm{II}$ by cutting in annuli. We sub-divide $\mathrm{II}$ further into two pieces; an inner piece and an outer piece.
The inner piece of $\mathrm{II}$ is estimated by
\begin{align}
\int_{B_{2^{-i-1}\delta}(p)} \frac{d \mu(y)}{\rho(x, y)^{Q-\alpha}} \leqslant c \frac{\mu\left(B_{2^{-i-1} \delta}\right)}{\left(2^{-i-1} \delta\right)^{Q-\alpha}}\label{eq:16}.
\end{align}
The outer piece of $\mathrm{II}$ we cut into annuli
$$
\begin{aligned}
& \sum_{k=i_0}^{i-1} \int_{{B_{2^{-k+2}\delta}(p)} \backslash B_{2^{-k+1} \delta}(p)} \frac{d \mu(y)}{\rho(x, y)^{Q-\alpha}} \\
& \leqslant c \sum_{k=i_0}^{i-1} \frac{\mu\left(B_{2^{-k+2} \delta}(p)\right)}{\left(2^{-k}\delta\right)^{Q-\alpha} }.
\end{aligned}
$$
Using \cref{eq:12} with $\varepsilon=0$ and \cref{eq:16}.
\begin{align}
\mathrm{II}  &\leqslant c \sum_{k=i_0}^{i-1} \frac{\left(2^{-k} \delta\right)^d}{\left(2^{-k} \delta\right)^{Q-\alpha}}+\frac{c}{\left(2^{-i-1} \delta\right)^{Q-\alpha-d}}. \nonumber\\
& \leqslant \frac{c}{\rho(x, p)^{Q-\alpha-d} }.\label{eq:17}
\end{align}
We estimate $\mathrm{III}$. Define
$$
E_i^\lambda=\left\{x \in {\omega}_\delta^i: \int_{\Omega_i^\delta} \frac{d\mu(y)}{\rho(x, y)^{Q-\alpha}} \geqslant \lambda 2^{i(Q-\alpha-d)}\right\},
$$
where $\lambda>0$ is fixed. By the definition
of capacity ($\mu$ is a ``trial" measure)
$$
C^\alpha\left(E_i^\lambda, \Omega_i^\lambda\right) \leqslant \frac{\mu\left(\Omega_i^\delta\right)}{\lambda 2^{i(Q-\alpha-d)}}.
$$
By \cref{eq:12},
$$
C^\alpha\left(E_i^\lambda, \Omega_i^\lambda\right)\leqslant c \frac{\left(2^{-i+2} \delta\right)^{d+\varepsilon}}{\lambda 2^{i(Q-\alpha-d)}} \text {. }
$$
By \cref{lemma-5,lemma-6}, we have
$$
C^\alpha\left(\partial B_{2^{-i} \delta}, B_{2^{-i+1} \delta}\right)=c\left(2^{-i} \delta\right)^{Q-\alpha}, c>0 \text {. }
$$
Thus
$$
\sum_{i \geqslant i_0} \frac{C^\alpha\left(E_i^\lambda, \Omega_i^\delta\right)}{C^\alpha\left(\partial B_{2^{-i}\delta}, B_{2^{-i+1}\delta}\right)} \leqslant \frac{c}{\lambda} \sum_{i \geqslant i_0} 2^{-i \varepsilon}<\infty.
$$
Thus if $x \notin \bigcup_i E_i^\lambda=E$, and $x \in \omega_\delta^i$
\begin{align}
\mathrm{III} \leqslant \lambda \cdot 2^{i(Q-\alpha-d)} \leqslant \frac{c}{\rho(x, p)^{Q-\alpha-d} \text {}} \text {.}\label{eq:18}
\end{align}
Collecting the estimates for $\mathrm{I}, \mathrm{II}, \mathrm{III}$, ie. \cref{eq:15,eq:17,eq:18}
$$
\begin{aligned}
& \int_{\Omega} \frac{d \mu}{\rho(x, y)^{Q-\alpha}} \leqslant \frac{c}{\rho(x, p)^{Q-\alpha-d}},
\end{aligned}
$$
if $x\notin E=\cup_i E_i^\lambda$.
\end{proof}
We now prove the main theorem.
\begin{proof}[Proof of \cref{theorem-1}]
By taking a connected sum of $M$ with a manifold of very negative Yamabe constant $M_1, M \# M_1$ can be assumed to have background scalar curvature $R \leqslant 0$. By \cref{lemma-4}, for $S \subseteq M$ and $S$ closed,
$$
-\Delta_b u=\mu \text { in } D,
$$
where $D$ is some neighborhood of $S$, and $\mu$ is a finite, Radon measure that is absolutely continuous with respect to $\theta \wedge(d \theta)^n$. Thus we have, via Green's
representation,
\begin{align}
\label{eq:19}
u=\int_D G(x, y) d \mu(y)+h(x),
\end{align}
where $h$ is harmonic with respect to $\Delta_b$. Now
$$
0 \leqslant G(x, y) \leqslant \frac{C}{\rho(x, y)^{Q-2}}.
$$ (cf.~\cites{nagel1985balls,bonfiglioli2007stratified}).
Thus from \cref{eq:19},
\begin{align}
u(x) \leqslant & \int_D \frac{d \mu(y)}{\rho(x, y)^{Q-2}}+h(x) \nonumber\\
& =R_\mu^{2, D}(x)+h(x) .\label{eq:20}
\end{align}
By contradiction assume Hausdorff $\dim_{\mathcal{H}}(S)=d>\frac{Q-2}{2}$. We know that $d\leq Q-2$. Hence, from \cref{lemma-7} and \cref{eq:20}, there exists $p \in S$ such that
$$
u(x) \leqslant \frac{c}{\rho(x, p)^{Q-2-d}}.
$$
We obtain,
\begin{align}
u(x)^{\frac{2}{Q-2}} \leqslant \frac{c}{\rho(x, p)^{2\frac{Q-2-d}{Q-2}}} .\label{eq:21}
\end{align}
Set
$$
\tau=\frac{2(Q-2-d)}{Q-2}=2-\frac{2 d}{Q-2}<1.
$$
Thus the curve length can be computed using \cref{eq:21} and using the metric
$$
\begin{aligned}
& \hat{\theta}=u^{\frac{4}{Q-2}} \theta
\end{aligned}
$$
we obtain the metric distance to the singular set is bounded by
$$
\begin{aligned}
\leqslant \int_\gamma \frac{d s}{s^\tau}<+\infty.
\end{aligned}
$$
This shows  $M \backslash S$ is not complete, which is a contradiction. This ends the proof.
\end{proof}

\section{Injectivity of the developing map}

Consider $(M^3,J,\theta)$ a three dimensional pseudo-hermitian CR manifold that is spherical and is compact with no boundary. Let $\tilde{M}$ the universal cover for $M$. Then from \cite{CY} we see one obtains a CR immersion
                              $$\Phi: \tilde{M}\to S^3.$$
This map is called the developing map. Our goal in this section is to establish under certain assumptions that this map is injective. Already in \cite{CY}, $\Phi$ was shown to be injective, under the requirement that the minimal integrability exponent of the Green function at infinity for the conformal sub-laplacian $L_\theta$, which we denote by $s(\tilde{M})$ satisfies $s(\tilde{M})<1$. Here we will prove injectivity under the hypotheses of the main theorem, Theorem 1.1 of \cite{Malchiodi}, which guarantees the positive mass theorem holds. We recall some definitions. First we have the Yamabe constant $Y(M)$ see \cite{Chanillo} and \cite{Malchiodi} eqn. $(1)$. We also have the CR Paneitz operator for which we refer to \cite{Chanillo}, \cite{Malchiodi}. We denote the CR Paneitz operator by $P$. We also need the conformal sub-laplacian,
          $$L_\theta= -4\Delta_b+R_\theta.$$

\begin{theorem}
\label{theorem-2}
Let $\left(M^n, \theta, J\right)$ be a compact pseudo-hermitian manifold. Assume $Y(M)>0$ and $P\geq 0$, that is the CR Paneitz operator is non-negative. Assume that $M$ is spherical. Let $\tilde{M}$ denote the universal cover for $M$. Then the developing map
                     $$\Phi: \tilde{M}\to S^3,$$
is injective.

\end{theorem}

\noindent
\textbf{Remark:}
\noindent
The hypotheses in our theorem above assures us as a consequence of the main theorem, Theorem 1.1 in \cite{Malchiodi}, that the CR mass is non-negative. We shall denote the CR mass by $A$ in what follows.

\begin{proof}[Proof of \cref{theorem-2}]For any $p\in \tilde{M}$, let $G_p$ denote the minimal positive Green function with pole at $p$ for the CR invariant sub-laplacian $L_\theta$. This Green function is constructed in Theorem 3.2 in \cite{CY}. Next we pull back the Green function from the sphere $S^3$ onto $\tilde{M}$, that is we form
                $$\tilde{G}=|\Phi(p)|^{-3}|\Phi^\prime|H_y\circ \Phi,$$
where $y=\Phi(p)$ and $H_y$ is the Green function for the CR invariant sub-laplacian $L_0$ associated to the standard contact form on $S^3$. We could use the Cayley map, from $S^3$ to the three dimensional Heisenberg group and even pull back the Green function from the Heisenberg group to $\tilde{M}$, see \cite{CY}. The poles of $\tilde{G}$ lie in the set $\Phi^{-1}(y)$ and thus $p$ is one among many poles. Thus to prove $\Phi$ is injective we need to show $\Phi^{-1}(y)=p$. To establish this, we proceed using the strategy in \cite{CY},\cite{SYdevelop} and in particular Theorem 3.5 in \cite{Schoenbook}. Thus we form
               $$v= \frac{G_p}{\tilde{G}}.$$
From \cite{CY} we note $0<v\leq 1$ and $v$ is smooth. Transferring to the Heisenberg group using the Cayley map, we note by $(4.3)$ of \cite{CY} that
                 $$\bar{\Delta_b}v=0,$$
where $\bar{\Delta_b}$ is the sub-laplacian on the Heisenberg group. The Cayley map is used to establish this fact in \cite{CY}, pg. 200. Thus $v$ is a harmonic function on the Heisenberg group to which Bony's maximal principle applies. We work in asymptotically flat coordinates, see \cite{CY} and \cite{Malchiodi}. We use the notation $y=(z,t)$ for points on the Heisenberg group and set
$\rho(y)=\rho(z,t)=\left(|z|^4+t^2\right)^{\frac{1}{4}}$ to be the metric on the Heisenberg group. By contradiction assume $v<1$. Since $\rho^{-2}$ is a solution to $\bar{\Delta_b}$ outside $y=0$, by the maximum principle we get for $y$ outside a ball $B_R(0)$ centered at the origin $(0,0)$ of the Heisenberg group
       $$1-v(y)\geq \left(\min_{\partial B_R(0)} (1-v)\right)R^2\rho\left(y\right)^{-2}.$$
Thus,
              $$v(y)\leq 1-c_0R^2\rho\left(y\right)^{-2},\quad c_0>0.$$
But this contradicts the positive mass theorem. In particular Lemma $4.1$ and Lemma $5.1$ in \cite{CY} and the positivity of the mass $A$ which follows from the main Theorem $1.1$ of \cite{Malchiodi} yields with $A> 0$,
                 $$v(y)=1+A\rho(y)^{-2}+O\left(\rho(y)^{-3}\right).$$
This ends the proof.
\end{proof}

\bibliographystyle{plain}
\bibliography{bib.bib}
\end{document}